\documentclass[reqno]{amsart}

\usepackage{amsmath,amssymb,amscd, accents} \usepackage{graphicx}
\DeclareGraphicsExtensions{.eps} \usepackage{mathrsfs}
\usepackage[mathcal]{eucal} 
\usepackage{esint}

\newtheorem{Thm}{Theorem}{\bfseries}{\itshape}
\newtheorem*{Thm*}{Theorem}{\bfseries}{\itshape}
\newtheorem{Cor}{Corollary}{\bfseries}{\itshape}
\newtheorem{Prop}[Cor]{Proposition}{\bfseries}{\itshape}
\newtheorem{Lem}[Cor]{Lemma}{\bfseries}{\itshape}
\newtheorem*{Lem*}{Lemma}{\bfseries}{\itshape}
{\bfseries}{\itshape}
{\bfseries}{\itshape}
\newtheorem{Def}[Cor]{Definition}{\bfseries}{\rmfamily}
\newtheorem{Ex}[Cor]{Example}{\scshape}{\rmfamily}
{\scshape}{\rmfamily}
{\bfseries}{\itshape}

\renewcommand\ge{\geqslant} \renewcommand\le{\leqslant}
\let\tildeaccent=\~ \let\hataccent=\^
\renewcommand\~[1]{\widetilde{#1}}

\def\<{\left<} \def\>{\right>} \def\({\left(} \def\){\right)}

\def\abs#1{\left\vert #1 \right\vert} \def\norm#1{\left\Vert #1
  \right\Vert} 

\let\parasymbol=\S \def\secref#1{\parasymbol\ref{#1}}

\let\polishL=l \def\Zoladek.{\.Zol\c adek}

 \def\Im{\operatorname{Im}}

\def\etc.{\emph{etc}.}
 \def\Sing{\operatorname{Sing}}

\def\:{\colon} \def\R{{\mathbb R}} \def\C{{\mathbb C}}  \def\N{{\mathbb N}} \def\Q{{\mathbb Q}}

\let\PolishL=\L % remember polish L
\def\L{{\mathbb L}}

 \def\e{\varepsilon} \def\S{\varSigma}
\def\l{\lambda}   
 
 \def\Lojas.{\PolishL ojasiewicz}
 \def\cH{{\mathcal H}}
 \def\cR{{\mathcal R}}

  \def\cR{{\mathcal R}}
\def\cI{{\mathcal I}}  
 \def\cD{{\mathcal D}}
\def\cM{{\mathcal M}}
\def\cO{{\mathcal O}}
\def\cC{{\mathcal C}}

\def\cV{{\mathcal V}}

\def\rest#1{{\vert_{#1}}}

\def\clo{\operatorname{Clo}}

\def\supp{\operatorname{supp}}

\def\id{\operatorname{id}}

\def\alg{{\mathrm{alg}}}
\def\trans{{\mathrm{trans}}}
\def\an{{\mathrm{an}}}

\def\vf{{\mathbf f}}
\def\vg{{\mathbf g}}
\def\vx{{\mathbf x}}
\def\vz{{\mathbf z}}
\def\vp{{\mathbf p}}

\begin{document}

% +Title
\title[The Pila-Wilkie theorem for subanalytic families]{The
  Pila-Wilkie theorem for subanalytic families: a complex analytic
  approach}

\author{Gal Binyamini} \author{Dmitry Novikov}
\address{Weizmann Institute of Science, Rehovot, Israel}
\email{galbin@gmail.com}

\begin{abstract}
  We present a complex analytic proof of the Pila-Wilkie theorem for
  subanalytic sets. In particular, we replace the use of $C^r$-smooth
  parametrizations by a variant of Weierstrass division.  
\end{abstract}
%% -Title
\maketitle
\date{\today}

\section{Introduction}

\subsection{Statement of the main results}

For a set $A\subset\R^m$ we define the \emph{algebraic part} $A^\alg$
of $A$ to be the union of all connected semialgebraic subsets of $A$
of positive dimension. We define the \emph{transcendental part}
$A^\trans$ of $A$ to be $A\setminus A^\alg$.

Recall that the \emph{height} of a (reduced) rational number
$\tfrac a b\in\Q$ is defined to be $\max(|a|,|b|)$. For a vector $x$
of rational numbers we denote by $H(x)$ the maximum among the heights
of the coordinates. For a set $A\subset\C^m$ we denote the set of
$\Q$-points of $A$ by $A(\Q):=A\cap\Q^m$ and denote
\begin{equation}
  A(\Q,H) := \{x\in A(\Q):H(x)\le H\}.
\end{equation}
For $A\subset\R^{m+n}$ and $y\in\R^n$ we denote the $y$-fiber of $A$
by
\begin{equation}
  A_y\subset\R^m, \qquad A_y := \{ x\in\R^m : (x,y)\in A \}.
\end{equation}
The following is our main result.
\begin{Thm}\label{thm:main}
  Let $A\subset\R^{m+n}$ be a bounded subanalytic set and $\e>0$.
  There exists an integer $N(A,\e)$ such that
  \begin{equation}
    \#(A_y)^\trans(\Q,H) \le N(A,\e) H^\e.
  \end{equation}
\end{Thm}

The result of Theorem~\ref{thm:main} is not new. In fact, it was
conjectured in \cite[Conjecture~1.2]{pila:subanalytic-dilation} and
proved, in a much more general form, in the work of Pila and Wilkie
\cite{pila-wilkie}, where the same result is shown to hold for any $A$
definable in an O-minimal structure. Our goal in the present paper is
to develop an alternative complex analytic approach to this theorem.
In particular, while the proof of \cite{pila-wilkie} requires the use
of $C^r$ parametrizations of subanalytic sets, we are able to carry
out the arguments completely within the analytic category.
In~\secref{sec:motivation} we present some motivations for our
approach. In~\secref{sec:bpw-review} we briefly review the method of
Bombieri-Pila-Wilkie, and in particular explain the point at which
$C^r$-parametrizations are required. In~\secref{sec:decomp-intro} we
give an outline of the complex-analytic approach developed in this
paper and explain how it avoids the use of $C^r$-parametrizations.

\subsection{Motivation}
\label{sec:motivation}

There are two directions in which one might hope to improve
the Pila-Wilkie estimate $\#A^\trans(\Q,H)\le N(A,\e) H^\e$:
\begin{itemize}
\item Effective estimates: one may hope to obtain effective
  estimates for the constant $N(A,\e)$ in terms of the complexity
  of the equations/formulas used to define $A$.
\item Sharper asymptotics: one may hope to improve the asymptotic
  dependence on $H$ if $A$ is definable is a suitably tame structure.
  As a notable example, the Wilkie conjecture states that if $A$ is
  definable in $\R_{\exp}$ then
  $\#A^\trans(\Q,H) = N(A)(\log H)^{\kappa(A)}$.
\end{itemize}
Both of these directions have been considered in the literature, see
e.g. \cite{butler,jones-thomas,pila:exp-alg-surface,pila:pfaff}.
However, as discussed in~\secref{sec:bpw-review}, the proof of the
Pila-Wilkie theorem in arbitrary dimensions requires the use of
$C^r$-reparametrizations, whose complexity is difficult to control
even in the semialgebraic case. For this reason, most of the work (to
our knowledge) has been restricted to $A$ of (real) dimension one or
two.

Our primary goal in this paper is to develop an approach that replaces
the use of $C^r$-parametrization by direct considerations on the local
complex-analytic geometry of $A$. In a closely related preprint
\cite{me:rest-wilkie} we use this approach to prove the Wilkie
conjecture for sets definable using the \emph{restricted} exponential
and sine functions. We believe that this approach may also play a
significant role in advancing toward an effective version of the
Pila-Wilkie theorem for Noetherian functions.

Finally, while in this paper we work in the complex analytic setting,
our arguments are essentially algebraic -- tracing to the Weierstrass
preparation and division theorems. One may hope that such an approach
could allow a more direct generalization to different algebraic
contexts where the analytic notion of $C^r$-parametrization may be
more difficult to recover. In particular we consider it an interesting
direction to check whether the method developed in this paper can
offer an alternative approach to the work of Cluckers, Comte and
Loeser \cite{ccl:nonarchimedean-pw} on non-archimedean analogs of the
Pila-Wilkie theorem. We remark in this context that in our primary
model-theoretic reference \cite{denef-vdd} the complex-analytic and
$p$-adic contexts are treated in close analogy.

\subsection{Exploring rational points following Bombieri-Pila and Pila-Wilkie}
\label{sec:bpw-review}

\subsubsection{The case of curves}
\label{sec:bpw-curves}
Let $X\subset\R^2$ be compact irreducible real-analytic curve.
Building upon earlier work by Bombieri and Pila \cite{bombieri-pila},
Pila \cite{pila:density-Q} considered the problem of estimating
$\#X(\Q,H)$. More specifically, he showed that if $X$ is
transcendental then for every $\e>0$ there exists a constant $C(X,\e)$
such that $\#X(\Q,H)\le C(X,\e)H^\e$. Bombieri and Pila's method
involves constructing a collection of $H^\e$ hypersurfaces $\{H_k\}$
of degree $d=d(\e)$ such that $X(\Q,H)$ is contained in $\cup_k H_k$.
We briefly recall the key idea, starting with the notion of an
interpolation determinant.

Suppose first that $X$ can be written as the image of an analytic map
$\vf=(f_1,f_2):[0,1]\to X$ (the general case will be treated later by
subdivision). For simplicity we will suppose that $f_1,f_2$ extend to
holomorphic functions in the disc of radius $2$ around the origin,
with absolute value bounded by $M$.

Let $\vg:=(g_1,\ldots,g_\mu)$ be a collection of functions and
$\vp:=(p_1,\ldots,p_\mu)$ a collection of points. We define the
\emph{interpolation determinant}
\begin{equation}
  \Delta(\vg,\vp) := \det (g_i(p_j))_{1\le i,j\le \mu}.
\end{equation}
Let $d\in\N$ and set $\mu=d(d+1)/2$, the dimension of the space of
polynomials in two variables of degree at most $d$. We define the
\emph{polynomial interpolation determinant} of degree $d$ to be
\begin{equation}
  \Delta^d(\vf,\vp) := \Delta(\vg,\vp), \qquad \vg=(f_1^kf_2^l : k,l\in\N, k+l\le d).
\end{equation}
Note that $\Delta^d(\vf,\vp)=0$ if and only if the points
$\vf(p_1),\ldots,\vf(p_\mu)$ all lie on a common algebraic
hypersurface of degree at most $d$. More generally, if $P\subset I$ and
$\Delta^d(\vf,\vp)=0$ for every $\vp\subset P$ then the points
$\vf(p):p\in P$ all lie on a common algebraic hypersurface of degree
at most $d$. 

Let $H\in\N$. Our goal is to construct a collection of algebraic
hypersurfaces $H_k$ whose union contains $X(\Q,H)$. By the above, it
will suffice to subdivide $I$ into intervals $I_k$ such that for any
$\vp\subset I_k$, if $H(\vp)\le H$ then $\Delta^d(\vf,\vp)=0$. We
begin with the following key estimate on the polynomial interpolation
determinant.

\begin{Lem}[\protect{cf. Lemma~\ref{lem:id-upper-bd}}]\label{lem:intro-id-upper-bd}
  Let $I\subset[0,1]$ be an interval of length $\delta<1/2$, and
  $\vp=(p_1,\ldots,p_\mu)\in I^\mu$. Then
  \begin{equation}\label{eq:intro-upper-bd}
    |\Delta^d(\vf,\vp)| \le \mu! (2\mu+2)^\mu M^{d\mu} \delta^{\mu^2/2}.
  \end{equation}
\end{Lem}
\begin{proof}
  By translation we may suppose that $I=[-\delta,\delta]$ and that
  $f_1,f_2$ are holomorphic in the unit disc $D\subset\C$ with
  absolute value bounded by $M$. Denote by $\norm\cdot$ the maximum
  norm on the disc of radius $\delta$ around the origin.

  Every function in $\vg$ is holomorphic in $D$ with absolute value
  bounded by $M^d$. Consider the Taylor expansions
  \begin{align}\label{eq:intro-g-taylor}
    g_i &= \sum_{j=0}^{\mu-1} m_j(g_i)+R_\mu(g_i) & m_j(g_i)&:=c_{i,j} x^j & i&=1,\ldots,\mu.
  \end{align}
  and $R_j(g_i)$ are the Taylor residues. From the Cauchy estimates we have
  \begin{align}\label{eq:intro-g-taylor-estimates}
    \norm{m_j(g_i)}&\le M^d \delta^j & \norm{R_\mu(g_i)}&\le 2M^d \delta^\mu
  \end{align}
  Expand the determinant $\Delta^d(\vf,\vp)$ by linearity
  using~\eqref{eq:intro-g-taylor}, to obtain a sum of $(\mu+1)^\mu$
  summands with each $g_i$ replaced by either $m_{j_i}(g_i)$ or
  $R_\mu(g_i)$. Note that any summand where two different indices
  $j_k,j_l$ agree vanishes identically since the corresponding
  functions $m_{j_k}(g_k),m_{j_l}(g_l)$ are linearly dependent.
  Therefore any non-zero summand must contain a term of order at least
  one, a term of order at least two, and so on. Then an easy
  computation using~\eqref{eq:intro-g-taylor-estimates} and the
  Laplace expansion for each determinant
  gives~\eqref{eq:intro-upper-bd}.
\end{proof}

Let $I,\vp$ be as in Lemma~\ref{lem:intro-id-upper-bd} and suppose
$\vf(p_1),\ldots,\vf(p_\mu)\in X(\Q,H)$. Using the bounded heights of
$\vf(p_j)$ one proves (cf. Lemma~\ref{lem:id-lower-bd}) that either
$\Delta^d(f,\vp)=0$ or
\begin{equation}\label{eq:intro-lower-bd}
  |\Delta^d(f,\vp)| \ge H^{-O(d^3)}.
\end{equation}
Comparing~\eqref{eq:intro-upper-bd} and~\eqref{eq:intro-lower-bd} and
recalling that $\mu\sim d^2$ we have either $\Delta^d(f,\vp)=0$ or
\begin{equation}
  H^{-O(d^3)} \le |\Delta^d(f,\vp)| \le 2^{O(d^3)} \delta^{\Omega(d^4)}
\end{equation}
where we treat $M$ as $O(1)$. Thus if $\delta=H^{-\Omega(1/d)}$ then
$\Delta^d(f,\vp)$ must vanish for any $\vp$ as above. Thus as
explained above, \emph{all} points $\vf(p)\in X(\Q,H)$ with $p\in I$
belong to a single algebraic hypersurface $H_k\subset\R^2$ of degree
at most $d$.

Fix $\e>0$ and subdivide $I$ into $H^\e$ subintervals $I_k$ of length
$\delta=H^{-\e}$. Then by the above all points of $X(\Q,H)$ belong to
a union of $H^\e$ hypersurfaces $H_k\subset\R^2$ of degree
$d=O(1/\e)$. If $X$ is irreducible and transcendental then it
intersects each $H_k$ properly, and number of intersections between
$X$ and $H_k$ is uniformly bounded by some constant $C(X,d)$ depending
only on $X$ and $d$ (for instance by Gabrielov's theorem). Thus we
have $\#X(\Q,H)\le C(X,\e) H^\e$.

To handle the case of a general compact irreducible analytic curve
$X\subset\R^2$ we note that any such curve can be covered by images of
analytic maps $\vf:[0,1]\to X$ and the preceding arguments apply.

\subsubsection{Higher dimensions}
\label{sec:bpw-higher-dim}
It is natural to attempt to generalize the proof
of~\secref{sec:bpw-curves} to sets $X\subset\R^m$ of dimension
$\ell>1$ by induction over $\ell$. Namely, the
estimates~\eqref{eq:intro-upper-bd} and~\eqref{eq:intro-lower-bd} can
be generalized in a relatively straightforward manner, replacing the
map $\vf:(0,1)\to X$ by an arbitrary analytic map
$\vf:(0,1)^\ell\to X$ parametrizing an $\ell$-dimensional set $X$. One
similarly obtains $H^\e$ hypersurfaces $H_k$ of some fixed degree
$d=d(\e)$ such that
\begin{equation}
  X(\Q,H) \subset \bigcup_k X\cap H_k.
\end{equation}
One would then seek to continue treating each intersection $X\cap H_k$
by induction on the dimension. However, at this point a problem
arises: even if the original set $X$ was parametrized by an analytic
map $\vf:(0,1)^\ell\to X$ it is not clear that $X\cap H_k$ could be
parametrized in a similar manner. Moreover, if one does obtain a
parametrization for each intersection $X\cap H_k$ then the induction
constant $C(X\cap H_k,\e)$ would now depend on the specific
parametrizations chosen for $X\cap H_k$, and one must show that these
constants are uniformly bounded over all $H_k$ of the given degree
$d$.

In fact, it is not always possible to choose analytic, or even
$C^\infty$-smooth, para\-metrizations for the fibers of a family in a
uniform manner -- even for semialgebraic families of curves. This
fundamental limitation was observed in the work of Yomdin
\cite{yomdin:entropy}. Consider for example the family of hyperbolas
$X_\e := (-1,1)^2\cap\{x^2-y^2=\e\}$. If one attempts to write $X_\e$
as a union of images $\Im\phi_j$ for $C^\infty$-smooth functions
$\phi_1,\ldots,\phi_{N_\e}:(0,1)\to X_\e$ with the maximum norms of
the derivatives of every order bounded by $1$, then $N_\e$ necessarily
tends to infinity as $\e\to0$. Thus it would not be possible to
parametrize all fibers of this family in a uniform manner and apply
to them the methods of~\secref{sec:bpw-curves}.

Surprisingly, a theorem due to Yomdin and Gromov
\cite{yomdin:entropy,yomdin:gy,gromov:gy} states that one can recover
the uniformity of $N_\e$ if one replaces the $C^\infty$ condition by
$C^r$-smoothness for a fixed $r$, at least for semialgebraic families.
In \cite{pila-wilkie} Pila and Wilkie generalized this result to the
O-minimal setting. Namely, they show
\cite[Corollary~5.1]{pila-wilkie} that for any set $X\subset(0,1)^m$
of dimension $\ell$ definable in an O-minimal structure and any
$r\in\N$, one can cover $X$ by images of $C^{r}$-maps
$\phi_1,\ldots,\phi_N:(0,1)^\ell\to X$ where $N=N(X,r)$ and every
$\phi_j$ has $C^{r}$-norm bounded by $1$ in $(0,1)^\ell$. Moreover,
\cite[Corollary~5.2]{pila-wilkie} if $X$ varies in a definable family
(and $r$ is fixed) then $N$ can be taken to be uniformly bounded over
the entire family.

One can now check that in the proof sketched
in~\secref{sec:bpw-curves} the analyticity assumption can be replaced
by $C^r$-smoothness (with bounded norms) for sufficiently large
$r=r(\e)$. The intersections $X\cap H_k$ can all be seen as fibers of
a single definable family by adding parameters for the coefficients of
$H_k$. One can thus parametrize each intersection $X\cap H_k$ by a
uniformly bounded number of $C^r$ maps with unit norms, and the
induction step can be carried out as sketched above.

Understanding the behavior of the parametrization complexity $N(X,r)$
in terms of the geometry of the set $X$ and the smoothness order $r$
is a highly non-trivial problem, even in the original context of the
Yomdin-Gromov theorem where $X$ is semialgebraic, and certainly in the
context of the O-minimal analog. It is this difficulty that prompted
us to seek a more direct approach for resolving the problem of
uniformity over families.

\subsection{An approach using holomorphic decompositions}
\label{sec:decomp-intro}

We return to the case of a compact irreducible real-analytic curve
$X\subset\R^2$. Let $p\in X$ and consider $(X,p)$ as the germ of a
complex-analytic curve. Then by Weierstrass preparation $X$ can be
written locally (up to a linear change of coordinate) in the form
\begin{equation}
  X=\{h=0\}, \qquad h(x,y)=y^\nu+a_{\nu-1}(x)y^{d-1}+\cdots+a_0(x)
\end{equation}
where $a_{\nu-1},\ldots,a_0$ are holomorphic in a neighborhood of $p$. By
Weierstrass division it follows that any $F$ holomorphic in a
neighborhood of $p$ can be written in the form
\begin{equation}\label{eq:intro-F-decomp}
  F = \sum_{i=0}^{\nu-1} \sum_{j=0}^\infty c_{i,j} y^i x^j + Q
\end{equation}
where $Q$ vanishes identically on $(X,p)$. Moreover, the coefficients
$c_{i,j}$ are bounded in terms of the norm of $F$ (cf.
Lemma~\ref{lem:preparation}). Let $\Delta_p\subset\C^2$ denote a
complex polydisc where the decomposition~\eqref{eq:intro-F-decomp} is
possible for any holomorphic $F$. We suppose for simplicity that
$\Delta_p$ has polyradius $1$ (the general case can be treated by
rescaling).

The polydiscs $\Delta_p$ serve as a replacement for the
parametrizations of~\secref{sec:bpw-curves}: we will show that one can
construct, in a completely analogous manner, $H^\e$ algebraic
hypersurfaces of degree $d=d(\e)$ containing all points of
$(\Delta_p\cap X)(\Q,H)$. Thus, instead of covering $X$ by images of
analytic parametrizing maps, we are led to the problem of covering $X$
by such ``good neighborhoods'' $\Delta_p$.

The key argument is the following analog of
Lemma~\ref{lem:intro-id-upper-bd}. Let $f_1,f_2$ be two holomorphic
functions on the polydisc of radius $2$ around $p$ and let their
absolute values be bounded by $M$.
\begin{Lem}[\protect{cf. Lemma~\ref{lem:id-upper-bd}}]
  \label{lem:intro-id-upper-bd-decomp}
  Let $D\subset \Delta_p$ be a polydisc of polyradius $\delta<1/2$, and
  $\vp=(p_1,\ldots,p_\mu)\in (D\cap X)$. Then
  \begin{equation}\label{eq:intro-upper-bd-decomp}
    |\Delta^d(\vf,\vp)| \le \mu! (\mu+1)^\mu M^{d\mu} \delta^{\Omega(\mu^2)}.
  \end{equation}
\end{Lem}
\begin{proof}
  The proof is essentially the same as that of
  Lemma~\ref{lem:intro-id-upper-bd}. We simply replace the Taylor
  expansion of the function $f_1,f_2$ by the
  expansions~\eqref{eq:intro-F-decomp} (and note that $Q$ vanishes on
  all points of $\vp$). In~\eqref{eq:intro-F-decomp} we have at most
  $\nu$ terms of each order $k$ (instead of one term of each order in
  the case of Taylor expansions), and this only introduces an extra
  factor depending on $\nu$ into the asymptotic
  $\delta^{\Omega(\mu^2)}$ in~\eqref{eq:intro-upper-bd-decomp}.
\end{proof}

We now proceed as in~\secref{sec:bpw-curves} taking $f_1=x$ and
$f_2=y$. In a similar manner, we can cover $\Delta_p\cap\R^2$ by
$H^\e$ polydiscs $D_k$ of polyradius $H^{-\e/2}$, and for each $D_k$
we find an algebraic hypersurface $H_k$ of degree $d=O(1/\e)$ such
that $(D_k\cap X)(\Q,H)\subset H_k$. Thus we see that
$(\Delta_p\cap X)(\Q,H)$ is contained in a union of $H^\e$ algebraic
hypersurfaces of degree $d$. Since $X$ is compact it may be covered by
finitely many of the polydiscs $\Delta_p$, and we finally see that
$X(\Q,H)$ is contained in a union of $O(H^\e)$ algebraic hypersurfaces
of degree $d$.

The main advantage of this approach becomes apparent when we consider
families of curves. Namely, unlike in the case of analytic
parametrizations, the argument above can be made uniform over analytic
families. To illustrate this consider again the family of hyperbolas
$X_\e := (-1,1)^2\cap\{x^2-y^2=\e\}$. The unit polydisc around the
origin $\Delta_0$ is a ``good neighborhood'' in the sense above,
\emph{uniformly for every $\e$}. Indeed, Weierstrass division with
respect to $y^2-x^2+\e$ is possible regardless of the value of $\e$
and the norms of the division remain bounded even as $\e\to0$. A
systematic application of Weierstrass division allows one to
generalize this example to an arbitrary family.

The purpose of this paper is to pursue this complex-analytic
perspective. In~\secref{sec:uniform-decomposition} we define the notion
of a \emph{decomposition datum} (see
Definition~\ref{def:decomposition}) generalizing the ``good
neighborhoods'' $\Delta_p$ above for complex analytic sets of
arbitrary dimension. We then prove in Theorem~\ref{thm:decomp} that
one can always cover (a compact piece of) a complex analytic set by
finitely many such polydiscs, and that this can be done uniformly over
analytic families (with a compact parameter space).
In~\secref{sec:interpolation-determinants} we show that in each such
polydisc the rational points of height $H$ can be described in analogy
with the Bombieri-Pila method of~\secref{sec:bpw-curves}.
In~\secref{sec:exploring} we prove a result analogous to the
Pila-Wilkie theorem for complex analytic sets of arbitrary dimension
(and their projections) by induction over dimension, in analogy with
the Pila-Wilkie method of~\secref{sec:bpw-higher-dim}. Finally
in~\secref{sec:subanalytic} we show that any bounded subanalytic set
can be complexified in an appropriate sense, and deduce
Theorem~\ref{thm:main} from the its complex-analytic version
Theorem~\ref{thm:exploring-projection}. The key technical tool for
this reduction is a quantifier-elimination result of Denef and van den
Dries \cite{denef-vdd}.

\section{Uniform decomposition in analytic families}
\label{sec:uniform-decomposition}

\subsection{Weierstrass division with norm estimates}

If $Z$ is a subset of a complex manifold $\Omega$ we denote by
$\cO(Z)$ the ring of germs of holomorphic functions in a neighborhood
of $Z$. If $Z$ is relatively compact in $\Omega$ we denote by
$\norm{\cdot}_Z$ the maximum norm on $\cO(\bar Z)$. We denote by
$\cO_\Omega$ the structure sheaf of $\Omega$, and if $X\subset\Omega$
is an analytic subset we denote by $\cI_X\subset\cO_\Omega$ its ideal
sheaf and by $\cI_{X,p}$ the germ of $\cI_X$ at $p$. Finally for an
ideal sheaf $\cI\subset\cO_\Omega$ we denote by $V(\cI)$ the analytic
set that it defines.

We say that a germ $f\in\C\{z_1,\ldots,z_n,w\}$ is regular of order
$d$ in $w$ if $f(0,w)=f_1(w)\cdot w^d$ with $f_1(0)\neq0$. For two
polydiscs $\Delta_v\subset\C$ and $\Delta_h\subset\C^n$, we say that
$\Delta:=\Delta_h\times\Delta_v$ is a \emph{Weierstrass polydisc} for
$f$ if $f(z,w)$ has exactly $d$ roots in $\Delta_v$ for any fixed
$z\in\bar\Delta_h$. In particular, $\Delta$ is a Weierstrass polydisc
for any sufficiently small $\Delta_v$ and sufficiently smaller
$\Delta_h$.

\begin{Lem}\label{lem:preparation}
  Let $f$ be regular of order $d$ in $w$, and
  $\Delta:=\Delta_h\times\Delta_v$ a sufficiently small Weierstrass
  polydisc for $f$. Then:
  \begin{enumerate}
  \item The map
    \begin{equation}
      \pi:\{f=0\}\cap \Delta \to \C^n, \qquad \pi(z,w)=z
    \end{equation}
    is finite.
  \item There exists a constant $C$ such that any $g\in\cO(\bar\Delta)$
    can be decomposed in the form
    \begin{equation}
      g = qf+ \sum_{k=0}^{d-1}  g_j w^j, \qquad g_j=g_j(z)
    \end{equation}
    with
    $\norm{g_j}_\Delta,\norm{q}_\Delta \le C\cdot\norm{g}_\Delta$.
  \end{enumerate}
\end{Lem}
\begin{proof}
  Since $\Delta$ is taken to be sufficiently small we may assume
  without loss of generality that $f$ is a Weierstrass polynomial of
  order $d$ in $w$. Then the first statement is classical and the
  second is the extended Weierstrass preparation theorem of
  \cite[II.D.~Theorem 1]{gr:analytic}.
\end{proof}

\subsection{Decomposition data}

We denote by $\vz$ a fixed system of affine coordinates on $\C^n$. We
say that $\vx$ is a \emph{standard} coordinate system on $\C^n$ if it
is obtained from $\vz$ by an affine unitary transformation. Given
$\vx$, we say that $(\Delta,\Delta')$ is a pair of polydiscs if
$\Delta\subset\Delta'$ are two polydiscs with the same center in the
$\vx$ coordinates.

For a co-ideal $\cM\subset\N^n$ and $k\in\N$ we denote by
\begin{equation}
  \cM^{\le k}:=\{ \alpha\in \cM : \abs{\alpha}\le k\}
\end{equation}
and by $H_\cM(k):=\#\cM^{\le k}$ its Hilbert-Samuel function. The function
$H_\cM(k)$ is eventually a polynomial in $k$, and we denote its degree
by $\dim\cM$.

If $(X,p)$ is the germ of an analytic set in $\C^n$ then there
exists a co-ideal $\cM$ with $\dim\cM=\dim X$ such that every
$F\in\cO_p$ can be decomposed as
\begin{equation}
  F = \sum_{\alpha\in\cM} c_\alpha \vz^\alpha + Q, \qquad Q\in\cO_p
\end{equation}
where $Q$ vanishes identically on $X$. For instance one may choose
$\cM$ to be the complement of the diagram of initial exponents of
$\cI_{X,p}$, in which case the claim above is a consequence of Hironaka
division. The following definition generalizes this notion from
the context of germs to the context of a fixed polydisc.

\begin{Def}\label{def:decomposition}
  Let $X\subset\C^n$ be a locally analytic subset, $\vx$ a standard
  coordinate system, $(\Delta,\Delta')$ a pair of polydiscs centered
  at the $\vx$-origin and $\cM\subset\N^n$ a co-ideal. We say that $X$
  admits \emph{decomposition} with respect to the \emph{decomposition
    datum}
  \begin{equation}
     \cD:=(\vx,\Delta,\Delta',\cM)
  \end{equation}
  if there exists a constant denoted $\norm\cD$ such that for every holomorphic
  function $F\in\cO(\bar \Delta')$ there is a decomposition
  \begin{equation}\label{eq:F-decomp}
    F = \sum_{\alpha\in\cM} c_\alpha \vx^\alpha + Q, \qquad Q\in\cO(\bar \Delta)
  \end{equation}
  where $Q$ vanishes identically on $X\cap \Delta$ and
  \begin{equation}\label{eq:F-decomp-norms}
    \norm{c_\alpha \vx^\alpha}_{\Delta} \le  \norm{\cD}\cdot \norm{F}_{\Delta'} \qquad \forall\alpha\in\cM.
  \end{equation}
  We define the \emph{dimension} of the decomposition datum, denoted
  $\dim\cD$ to be $\dim\cM$.
\end{Def}

%I guess it is ok to take minimum here but not entirely obvious,
% we don't really need minimum so maybe rephrase this somehow?
Since $\cH_\cM(k)$ is eventually a polynomial of degree $\dim\cM$, the
function $\cH_\cM(k)-\cH_\cM(k-1)$ counting monomials of degree $k$ in
$\cM$ is eventually a polynomial of degree $\dim\cM-1$. If
$\dim\cM\ge1$ we denote by $e(\cD)$ the minimal constant satisfying
\begin{equation}
   H_\cM(k)-H_\cM(k-1) \le e(\cD)\cdot L(\dim\cM,k), \qquad \forall k\in\N.
\end{equation}
where $L(n,k):=\binom{n+k-1}{n-1}$ denotes the dimension of the space
of monomials of degree $k$ in $n$ variables. In the case $\dim\cD=0$
the co-ideal $\cM$ is finite and we denote by $e(\cD)$ its size.

\begin{Ex}\label{ex:decomposition-dim0}
  Suppose $X$ admits decomposition with respect to the decomposition
  datum $\cD$, and $\dim\cD=0$. Then $N=\#(X\cap\Delta)$ is finite and
  satisfies $N\le e(\cD)$. Indeed, by~\eqref{eq:F-decomp} any
  polynomial on $\C^n$ can be interpolated on $X\cap\Delta$ by the
  $e(\cD)$ monomials of $\cM$. Since the linear space of polynomials
  restricted to $X\cap\Delta$ has dimension $N$ it follows that
  $N\le e(\cD)$.
\end{Ex}

\subsection{Decomposition data for analytic families}

If $X\subset\C^n$ is a locally analytic subset and $k\in\N$, we denote
by $X^{\le k}$ the union of the components of $X$ that have dimension
$k$ or less. Note that $X^{\le k}$ is locally analytic as well.

Let $\Omega\subset\C^n$ be an open subset and $\Lambda$ a complex
analytic space. We denote by $\pi_\Omega,\pi_\Lambda$ the projections
from $\Omega\times\Lambda$ to $\Omega,\Lambda$ respectively. For
$X\subset\Omega\times\Lambda$ and $\l\in\Lambda$ we denote the
$\l$-fiber of $X$ by
\begin{equation}
  X_\l\subset\Omega, \qquad X_\l := \{ p\in\Omega : (p,\l)\in X \}.
\end{equation}

The following theorem is our main result on uniform decomposition in
families. It says roughly that if one considers a compact piece of an
analytic family $X$, then each fiber $X_\l$ at every point $p$ admits
decomposition with respect to some decomposition datum $\cD$ with
$\dim\cD=\dim X_\l$, with the size of the polydisc $\Delta$ bounded
from below and $\norm\cD,e(\cD)$ bounded from above \emph{uniformly
over the (compact) family}.

\begin{Thm}\label{thm:decomp}
  Let $X\subset \Omega\times\Lambda$ be an analytic subset,
  $K\Subset\Omega\times\Lambda$ a compact subset and $k\in\N$. There
  exists a positive radius $r>0$ and constants $C_D,C_H>0$ with the
  following property. For any $(p,\l)\in K$ there exists a
  decomposition datum $\cD$ such that:
  \begin{enumerate}
  \item $\Delta=\Delta'$ is centered at $p$, and
    $B_r(p)\subset\Delta\subset\Omega$.
  \item $\dim\cM_i\le k$, $\norm\cD\le C_D$ and $e(\cD)\le C_H$.
  \item $(X_\l)^{\le k}$ admits decomposition with respect to $\cD$
  \end{enumerate}
\end{Thm}

We first consider the problem of constructing decomposition data of
dimension $k$ for fibers of a family $X$, under the assumption that
all fibers of $X$ have dimension bounded by $k$. This basic case
essentially reduces to Hironaka division. For completeness we give a
proof using Weierstrass division.

\begin{Lem}\label{lem:decomp-pure}
  Let $X\subset \Omega\times\Lambda$ be an analytic subset, $k\in\N$
  and suppose $\dim X_\l\le k$ for every $\l\in\Lambda$. Then for any
  $p\in\Omega$ and compact $K_\Lambda\Subset\Lambda$ a there exists a
  finite collection of decomposition data $\{\cD_i\}$ such that:
  \begin{enumerate}
  \item $\Delta_i=\Delta_i'$ is centered at $p$ and contained in $\Omega$.
  \item $\dim\cM_i\le k$.
  \item for every $\l\in K_\Lambda$ the fiber $X_\l$ admits
    decomposition with respect to some $\cD_i$.
  \end{enumerate}
\end{Lem}

\begin{proof}
  Let $\vz$ be a standard coordinate system centered at $p$. We
  proceed by induction on $n$. If $n=k$ then the claim holds with any
  choice of $\vx$, $\cM=\N^n$ and $\Delta=\Delta'$ any polydisc
  contained in $\Omega$. The expansion~\eqref{eq:F-decomp} is given by
  the usual Taylor expansion for $F$ around the origin with
  $Q\equiv0$. The inequality~\eqref{eq:F-decomp-norms} is given by the
  Cauchy estimates.

  Suppose $n>k$. By compactness it will suffice to prove the claim in
  a neighborhood of each $\l\in K_\Lambda$. Fix $\l_0\in K_\Lambda$.
  Since $\dim X_{\lambda_0}<n$ there exists $G\in\cI_{X,p}$ such that
  $G\rest{\l=\l_0}\not\equiv0$. By a unitary change of the
  $z$-coordinates we may suppose that $G$ is regular with respect to
  $z_n$, of some order $d$. Then by Lemma~\ref{lem:preparation} the
  map
  \begin{equation}
    \pi_n: V(G)\to \C^{n-1}\times\tilde\Lambda, \qquad \pi_n(z_1,\ldots,z_n,\lambda)=(z_1,\ldots,z_{n-1},\lambda)
  \end{equation}
  is finite when restricted to an appropriate polydisc
  $D=D_z\times D_\l$, where $D_z=D_h\times D_v$ and $D_h,D_\l$ are
  chosen to be sufficiently smaller than $D_v$. Then
  $Y:=\pi_n(X\cap D)$ is analytic in $D_h\times D_\l$ by the proper
  mapping theorem.
  
  Let $K_{\l_0}\subset D_\l$ be some compact neighborhood of $\l_0$.
  Since $\pi_n:X\to Y$ is finite we have $\dim Y_\l\le\dim X_\l\le k$
  for $\l\in K_{\l_0}$. Apply the inductive hypothesis with $Y$ for
  $X$, $K_{\l_0}$ for $K_\Lambda$ and $D_h$ for $\Omega$ to obtain a
  finite collection of decomposition data $\{\hat\cD_i\}$. We let
  \begin{align}
    \vx_i&:=(\hat\vx_i,z_n), \\
    \Delta_i=\Delta_i'&:=\hat\Delta_i\times D_v, \\
    \cM_i&:=\hat\cM_i\times\{0,\ldots,d-1\}.
  \end{align}
  Note that since $\hat\Delta_i\subset D_h$ and $D_\l$ are chosen to be
  sufficiently smaller than $D_v$, Lemma~\ref{lem:preparation} applies
  with the polydisc $\Delta_i\times D_\l$. Applying the lemma to
  $F(x,\l)\equiv F(x)$ we obtain a decomposition
  \begin{equation}\label{eq:F-x-expand}
    F = \sum_{j=0}^{d-1} z_n^j F_j + QG, \qquad F_j=F_j(z_1,\ldots,z_{n-1},\l)
  \end{equation}
  with $\norm{F_j}_{\Delta_i\times D_\l}= O_{\l_0}(\norm{F}_{\Delta_i})$.

  By construction $Y_\l$ admits decomposition with respect to some
  $\hat\cD_i$. Hence we may decompose the functions
  $F_j(\cdot)\equiv F_j(\cdot,\lambda)$ as
  \begin{equation}\label{eq:Fj-expand}
    F_j = \sum_{\alpha\in\hat\cM_i} c_{j,\alpha} \hat \vx_i^\alpha + Q_j,
    \qquad Q_j\in\cO(\overline{\hat\Delta_i})
  \end{equation}
  where
  \begin{enumerate}
  \item $\hat\cM_i\subset\N^{n-1}$ is a co-ideal and $\dim\hat\cM_i\le k$.
  \item $Q_j$ vanishes identically on $Y_\lambda\cap \hat\Delta_i$.
  \item We have
    \begin{equation}
      \norm{c_{j,\alpha} \hat\vx_i^\alpha}_{\hat\Delta_i}= O_{\l_0}(\norm{F_j}_{\hat\Delta_i})
      = O_{\l_0}(\norm{F}_{\Delta_i}).
    \end{equation}
  \end{enumerate}
  Plugging~\eqref{eq:Fj-expand} into~\eqref{eq:F-x-expand} we obtain
  the decomposition~\eqref{eq:F-decomp}.
\end{proof}

To observe the principal limitation of Lemma~\ref{lem:decomp-pure}
consider the family $X:=\{\l_1x=\l_2\}\subset\C_x\times\C^2$. The
fiber $X_{(0,0)}$ is one-dimensional while every other fiber is
zero-dimensional. We would like to produce decomposition data of
dimension zero for the fibers away from the origin, with constants
remaining uniformly bounded as we approach the origin. However
Lemma~\ref{lem:decomp-pure} only guarantees the existence of
decomposition data of dimension one. The following proposition
eliminates this limitation, producing for each fiber a decomposition
datum of the correct dimension. The idea of the proof is to use
blowings-up to avoid the jump in the dimension of the fiber. For
instance, the reader may observe that in the preceding example, after
blowing up the origin $\{\l_1=\l_2=0\}$ the strict transform
$\tilde X$ has only zero-dimensional fibers.

\begin{Prop}\label{prop:decomp}
  Let $X\subset \Omega\times\Lambda$ be an analytic subset and
  $k\in\N$. Then for any $p\in\Omega$ and compact
  $K_\Lambda\Subset\Lambda$ a there exists a finite collection of
  decomposition data $\{\cD_i\}$ such that:
  \begin{enumerate}
  \item $\Delta_i=\Delta_i'$ is centered at $p$ and contained in $\Omega$.
  \item $\dim\cM_i\le k$.
  \item for every $\l\in K_\Lambda$ the set $(X_\l)^{\le k}$ admits
    decomposition with respect to some $\cD_i$.
  \end{enumerate}
\end{Prop}

\begin{proof}
  Let $m:=\dim\Lambda$ and $d:=\max\{\dim X_\l:\l\in K_\Lambda\}$. We
  proceed by induction on $(m,d)$ with the lexicographic order. By
  compactness it will suffice to prove the claim in a neighborhood of
  each $\l\in K_\Lambda$. Fix $\l_0\in K_\Lambda$. Without loss of
  generality we may replace $X$ by its germ at $(p,\l_0)$ and
  $\Lambda,K_\Lambda$ by their germs at $\l_0$. For a sufficiently
  small germ we have (by semicontinuity of the dimension)
  $d=\dim X_{\l_0}$. If $d\le k$ then the claim follows by
  Lemma~\ref{lem:decomp-pure}, so we assume $d>k$.

  We may assume without loss of generality that $\Lambda$ is smooth.
  Indeed, otherwise let $\sigma:M\to\Lambda$ be a desingularization
  \cite{hironaka:resolution} of $\Lambda$ and
  $\tilde X:=X\times_\Lambda M$. Note that this does not change the
  pair $(m,d)$. Every fiber of $X$ is a fiber of $\tilde X$, and it
  suffices to prove the claim for the compact set
  $\sigma^{-1}(K_\Lambda)$.

  We may also assume without loss of generality that $\dim X<m+d$.
  Indeed, if $X$ has a component $X'$ of dimension $m+d$ then the
  fibers $X'_\l$ must have pure dimension $d$ so
  $(X'_\l)^{\le k}=\emptyset$. Thus it is enough to prove the claim
  for the union of the components of $X$ that have dimension strictly
  smaller than $m+d$.

  Since $d=\dim X_{\l_0}$ there exists an affine linear projection
  $\pi_d:\C^n\to\C^d$ such that
  \begin{equation}
    \pi = \pi_d\times\pi_\Lambda : (X,(p,\l_0))\to (\C^d\times\Lambda,(0,\l_0))
  \end{equation}
  is finite and hence $Y=\pi(X)$ is the germ of an analytic subset at
  $(0,\l_0)$. In particular $\dim Y=\dim X<m+d$ so $Y\neq\C^d\times\Lambda$. Then there
  exists a non-zero $G\in\cI_{Y,(0,\l_0)}$. Write
  \begin{equation}
    G = \sum_\alpha c_\alpha(\l) w^\alpha, \qquad (w,\l)\in\C^d\times\Lambda
  \end{equation}
  and let $I$ be the ideal generated by $\{c_\alpha\}$ in
  $\cO_{\Lambda,\l_0}$. Then the set $\cC:=V(I)\subset\Lambda$ is an
  analytic space of dimension strictly smaller than $m$, and the claim
  follows for any $\l\in\cC$ by induction on $m$. It remains to
  construct suitable decomposition data for any $\l\not\in\cC$.

  Let $\eta:\tilde\Lambda\to\Lambda$ denote the blowing up of $I$ and
  $X':=X\times_\Lambda\tilde\Lambda$. Let
  \begin{equation}
    \tilde X := \clo[X'\setminus (\C^n\times\eta^{-1}(\cC))]
  \end{equation}
  be the strict transform of $X$ (where $\clo$ denotes analytic
  closure). For any $\lambda\in\Lambda\setminus\cC$, the fiber $X_\l$
  is also a fiber of $\tilde X$. Thus it will suffice to prove the
  claim for the family $\tilde X$ and the compact set
  $\eta^{-1}(K_\Lambda)$. Let $\tilde\l\in\tilde\Lambda$ and we will
  show that $\dim\tilde X_{\tilde\l}<d$, and the claim thus follows by
  induction on $d$.

  By definition of the blow-up $\eta$, the ideal
  $I\cO_{\tilde\Lambda,\tilde\l}$ is principal hence generated by some
  $c_\alpha$. Thus we may write $(\id\times\eta)^*G=c_\alpha\tilde G$
  where $\tilde G\in\cO_{\C^d\times\tilde\Lambda,(0,\tilde\l)}$ does
  not vanish identically on $\C^d\times\{\tilde\l\}$. Since
  $I=\<c_\alpha\>$ near $\tilde\l$ the strict transform satisfies
  $\tilde X\subset V((\pi_d\times\id)^*\tilde G)$ and thus
  $\pi_d(\tilde X_{\tilde\l})\subset \C^d\cap\{\tilde G=0\}$. The map
  $\pi_d\rest{\tilde X_{\tilde\l}}$ is finite, being the restriction
  of a finite map $\pi_d\rest{X_\l}$ for some $\l\in\Lambda$, and
  we conclude that $\dim\tilde X_{\tilde\l}<d$ as claimed.
\end{proof}

Finally we finish the proof of Theorem~\ref{thm:decomp}.

\begin{proof}[Proof of Theorem~\ref{thm:decomp}.]
  By compactness there exists a ball $B\subset\C^n$ such that
  $p+B\subset\Omega$ for every $(p,\l)\in K$. Let
  $\Lambda'=\Omega\times\Lambda$ and $K_{\Lambda'}=K$. Define
  \begin{equation}
    X' \subset B\times\Lambda', \qquad X':=\{ (q,(p,\l)) : q+p\in X_\l \}.
  \end{equation}
  By definition, $X'_{(p,\l)}=(-p)+X_\l$ in $B$. Apply
  Proposition~\ref{prop:decomp} to $X',K_{\Lambda'}$ with the point
  $q=0$ to obtain a finite collection of decomposition data
  $\{\cD_i\}$ with $\dim\cD_i\le k$. Let $C_D,C_H$ be the minimum
  among the corresponding parameters $\norm{\cD_i},e(\cD_i)$ and
  choose some $r>0$ such that $B_r(0)\subset\Delta_i$ for every $i$.

  Now let $(p,\l)\in K$. By Proposition~\ref{prop:decomp},
  $(X'_{(p,\l)})^{\le k}$ admits decomposition with respect to some
  $\cD_i$. Define $\cD$ as the $p$-translate of $\cD_i$, i.e.
  $\Delta=p+\Delta_i$ and $\vx=p+\vx_i$. Then
  $(X_\l)^{\le k}=(p+X'_{(p,\l)})^{\le k}$ admits decomposition with
  respect to $\cD$ as claimed.
\end{proof}

\section{Interpolation determinants and rational points}
\label{sec:interpolation-determinants}

Let $A\subset\C^n$ be a ball or polydisc around a point $p\in\C^n$ and
$\delta>0$. We let $A^\delta$ denote the $\delta^{-1}$-rescaling of
$A$ around $p$, i.e. $A^\delta:=p+\delta^{-1}(A-p)$.

Let $X\subset\C^n$ be an analytic subset and $\cD$ a decomposition
datum for $X$, and set $m:=\dim\cM$. We suppose $D$ is a polydisc in
the $\vx$ coordinates, centered at $p$ and $D^\delta\subset\Delta$.

\subsection{Norm estimates}

Below $\norm\cdot$ denotes $\norm\cdot_D$ and $\norm\cdot_\delta$
denotes $\norm\cdot_{D^\delta}$. We remark that
$\norm{\vx^\alpha}=\norm{\vx^\alpha}_\delta \delta^{|\alpha|}$.

\begin{Prop}
  Let $f\in\cO(\bar\Delta')$ and denote $M:=\norm{f}_{\Delta'}$. For
  every $k\in\N$ we have
  \begin{equation}\label{eq:F-R-decomp}
    f = \sum_{\alpha\in\cM^{<k}} m_\alpha(f) + R_k(f) + Q
  \end{equation}
  where $Q\in\cO(\bar\Delta)$ vanishes on $X\cap\Delta$ and
  \begin{equation}
    m_\alpha(f) = c_\alpha \vx^\alpha, \qquad R_k(f)=\sum_{\cM\ni|\alpha|\ge k} c_\alpha \vx^\alpha.
  \end{equation}
  Moreover,
  \begin{equation}\label{eq:F-R-decomp-norms}
    \norm{m_\alpha(f)} \le \norm\cD M \delta^{|\alpha|}, \qquad
    \norm{R_k(f)} \le \frac{\norm\cD e(\cD) L(m,k)}{(1-\delta)^m} M  \delta^k 
  \end{equation}
\end{Prop}
\begin{proof}
  The decomposition~\eqref{eq:F-R-decomp} is just~\eqref{eq:F-decomp}.
  Then~\eqref{eq:F-decomp-norms} gives
  \begin{equation}
     \norm{m_\alpha(f)} = \norm{m_\alpha(f)}_\delta \delta^{|\alpha|}
     \le \norm{m_\alpha(f)}_\Delta \delta^{|\alpha|} \le \norm\cD M \delta^{|\alpha|}
  \end{equation}
  where we used that fact that $D^\delta\subset\Delta$ in the middle
  inequality. Then
  \begin{multline}
    \norm{R_k(f)} \le \sum_{\cM\ni|\alpha|\ge k} \norm\cD M\delta^{|\alpha|}\le
    \norm\cD M e(\cD) \sum_{j=0}^\infty L(m,j+k)\delta^{j+k} \\
    \le \norm\cD e(\cD) M L(m,k)\delta^k \sum_{j=0}^\infty L(m,j)\delta^j \\
    = \frac{\norm\cD e(\cD) L(m,k)}{(1-\delta)^m} M \delta^k.
  \end{multline}
\end{proof}

\subsection{Interpolation determinants}

Let $\vf:=(f_1,\ldots,f_\mu)$ be a collection of functions and
$\vp:=(p_1,\ldots,p_\mu)$ a collection of points. We define the
\emph{interpolation determinant}
\begin{equation}
  \Delta(\vf,\vp) := \det (f_i(p_j))_{1\le i,j\le \mu}.
\end{equation}

In the asymptotic notations $\sim_m,O_m,\Omega_m$ below we use the
subscript $m$ to indicate that the implied constants depend only on
$m$. The following lemma and its proof are direct analogs of the
interpolation determinant estimates of \cite{bombieri-pila}. We remark
that in this paper we will not make explicit use of the estimates for
the constants $C,E$ in terms of $\norm\cD,e(\cD)$.

\begin{Lem}\label{lem:id-upper-bd}
  Assume $m>0$. Suppose $f_i\in\cO(\bar\Delta')$ with
  $\norm{f_i}_{\Delta'}\le M$ and $p_i\in D\cap X$ for
  $i=1,\ldots,\mu$. Assume $\delta<1/2$. Then
  \begin{equation}\label{eq:id-bound}
    \abs{\Delta(\vf,\vp)} \le (C \mu^3 M)^\mu \cdot \delta^{E\cdot\mu^{1+1/m}}
  \end{equation}
  where
  \begin{align}
    C &= O_m(\norm\cD e(\cD)^{1/m}), \\
    E &= \Omega_m(e(\cD)^{-1/m}).
  \end{align}
\end{Lem}
\begin{proof}
  We set
  \begin{equation}
    k := \max \{ j : \sum_{l=0}^j e(\cD) L(m,l) < \mu \}.
  \end{equation}
  Since $\sum_{l=0}^j L(m,l)=L(m+1,j)$ is a polynomial of degree $m$
  we have $k\sim_m (\mu/e(\cD))^{1/m}$.
  
  We consider the expansions~\eqref{eq:F-R-decomp} for each $f_i$ with
  $k$ as above,
  \begin{equation}\label{eq:fi-exp}
    f_i = \sum_{\alpha\in\cM^{<k}} m_\alpha(f_i) + R_k(f_i) + Q_i.
  \end{equation}
  We note that $Q_i$ vanishes identically on $X\cap\Delta$ and in
  particular at every $p_j$. By definition of $k$, the number of
  remaining terms in~\eqref{eq:fi-exp} does not exceed $\mu$. We
  expand $\Delta(\vf,\vp)$ by linearity with respect to each column.
  We thus obtain a sum of at most $\mu^\mu$ interpolation determinants
  $\Delta_I$ where each $f_i$ is replaced by either a monomial term
  $m_\alpha(f_i)$ or a residue term $R_k(f_i)$.
  By~\eqref{eq:F-R-decomp-norms} we have for $i=1,\ldots,\mu$ and for
  every $\alpha\in\cM^{<k}$
  \begin{equation}\label{eq:m-r-norms-combined}
    \begin{aligned}
      \norm{m_\alpha(f_i)} &\le C_0 \delta^{|\alpha|} \\ \norm{R_k(f_i)} &\le C_0 \delta^k
    \end{aligned}
    \qquad \text{where } C_0 := \frac{\norm\cD e(\cD) L(m,k)}{(1-\delta)^m}M
  \end{equation}
  We remark that these are estimates for the maximum norm in $D$, and
  in particular they bound the absolute value of
  $m_\alpha(f_i),R_k(f_i)$ at every point $p_j$.

  Note that if the same index $\alpha$ is repeated in two different
  columns of $\Delta_I$ then these columns are linearly dependent and
  $\Delta_I\equiv0$. Thus for every non-zero $\Delta_I$ we can have at
  most
  \begin{equation}
    H_\cM(j)-H_\cM(j-1) \le e(\cD) L(m,j)
  \end{equation}
  monomial terms of order $|\alpha|=j$. We now expand $\Delta_I$ by
  the Laplace expansion. By definition of $k$ and
  by~\eqref{eq:m-r-norms-combined} we conclude that for each
  $\Delta_I$ we have
  \begin{equation}
    |\Delta_I| \le \mu! C_0^\mu \delta^S, \qquad S(m,k):=\sum_{l=0}^k e(\cD) L(m,l)\cdot l.
  \end{equation}
  Since $L(m,l)\cdot l$ is a polynomial of degree $m$ in $l$, we
  conclude that $S(m,k)\sim_m e(\cD) k^{m+1}$.

  Plugging in $k\sim_m (\mu/e(\cD))^{1/m}$ we have
  \begin{align}
    C_0 = &O_m(\norm\cD e(\cD)^{1/m} \mu^{1-1/m} M), \\
    S(m,k) &\sim_m e(\cD)^{-1/m} \mu^{1+1/m}.
  \end{align}
  Summing over the (at most) $\mu^\mu$ determinants $\Delta_I$ we
  obtain~\eqref{eq:id-bound}.
\end{proof}

\subsection{Polynomial interpolation determinants}

Let $d\in\N$ and let $\mu$ denote the dimension of the space of
polynomials of degree at most $d$ in $m+1$ variables, $\mu=L(m+2,d)$.
Let $\vf:=(f_1,\ldots,f_{m+1})$ be a collection of functions and
$\vp:=(p_1,\ldots,p_\mu)$ a collection of points. We define the
\emph{polynomial interpolation determinant} of degree $d$ to be
\begin{equation}
  \Delta^d(\vf,\vp) := \Delta(\vg,\vp), \qquad \vg=(\vf^\alpha : \alpha\in\N^{m+1}, |\alpha|\le d).
\end{equation}
Note that $\Delta^d(\vf,\vp)=0$ if and only if there exists a
polynomial of degree at most $d$ in $m+1$ variables vanishing at the
points $\vf(p_1),\ldots,\vf(p_\mu)$.

\begin{Lem}\label{lem:id-lower-bd}
  Let $H\in\N$ and suppose that
  \begin{equation}
    H(f_i(p_j)) \le H \qquad
    \begin{aligned}
      i&=1,\ldots,m+1 \\
      j&=1,\ldots,\mu.
    \end{aligned}
  \end{equation}
  Then $\Delta^d(\vf,\vp)$ either vanishes or satisfies
  \begin{equation}
    \abs{\Delta^d(\vf,\vp)} \ge H^{-(m+1)d\mu}.
  \end{equation}
\end{Lem}
\begin{proof}
  Let $Q_{i,j}$ denote the denominator of $f_i(p_j)$ for
  $i=1,\ldots,m+1$ and $j=1,\ldots,\mu$. By assumption $Q_{i,j}\le H$.
  The row corresponding to $p_j$ in $\Delta^d(\vf,\vp)$ consists of
  rational numbers with common denominator dividing
  $Q_j:=\prod_i Q^d_{i,j}$. Factoring out $Q_j$ from each row we obtain
  a matrix with integer entries, whose determinant is either
  vanishing or at least one in absolute value. In the non-vanishing
  case we have
  \begin{equation}
    \abs{\Delta^d(\vf,\vp)} \ge \prod_{j=1}^\mu Q_j^{-1} \ge H^{-(m+1)d\mu}.
  \end{equation}
\end{proof}

Comparing Lemmas~\ref{lem:id-upper-bd} and~\ref{lem:id-lower-bd} we
obtain the following.

\begin{Prop}\label{prop:hypersurface-select}
  Let $M,H\ge2$, and suppose $f_i\in\cO(\bar\Delta')$ with
  $\norm{f_i}_{\Delta'}\le M$. Assume $\delta<1/2$. Let
  \begin{equation}
    Y = \vf(X\cap D) \subset \C^{m+1}.
  \end{equation}
  There exist a constant $C_1>0$ depending only on $m>0$ such that if
  \begin{equation}\label{eq:hypersurface-cond}
    -\log \delta > C_1 \frac{d^{-1}\log (\norm\cD e(\cD))+\log M+\log H}{(d/e(\cD))^{1/m}}
  \end{equation}
  then $Y(\Q,H)$ is contained in an algebraic hypersurface of degree
  at most $d$ in $\C^{m+1}$. The same conclusion holds for $m=0$ if instead
  of~\eqref{eq:hypersurface-cond} we assume $d\ge e(\cD)$.
\end{Prop}
\begin{proof}
  We consider first the case $m=0$. In this case according to
  Example~\ref{ex:decomposition-dim0} the number of points in
  $X\cap\Delta$ is bounded by $e(\cD)$. In particular this bounds the
  number of points in $Y$, and all the more in $Y(\Q,H)$. Thus the
  claim holds with any $d\ge e(\cD)$.
  
  Now assume $m>0$ and suppose toward contradiction that $Y(\Q,H)$ is
  not contained in an algebraic hypersurface of degree at most $d$ in
  $\C^{m+1}$. Then by standard linear algebra it follows that there
  exist $\vp=p_1,\ldots,p_\mu\in X\cap D$ such that
  $\{\vf(p_j):j=1,\ldots,\mu\}$ is a subset of $Y(\Q,H)$ and does not
  lie on the zero locus of a non-zero polynomial of degree $d$. Then
  $\abs{\Delta^d(\vf,\vp)}\neq0$, and from
  Lemmas~\ref{lem:id-upper-bd} and~\ref{lem:id-lower-bd} we have
  \begin{equation}
    H^{-(m+1)d\mu}\le \abs{\Delta^d(\vf,\vp)} \le (C \mu^3 M^d)^\mu \cdot \delta^{E\cdot\mu^{1+1/m}}.
  \end{equation}
  Takings logs and using $\mu\sim_m d^{m+1}$ we have
  \begin{multline}
    (d/e(\cD))^{1/m}d \log\delta \ge\\-\Omega_m\big(
    \log \norm\cD+m^{-1}\log e(\cD)+3(m+1)\log d +d \log M+(m+1)d\log H \big).
  \end{multline}
  Noting that
  \begin{equation}
    \frac{\log d}{d^{1+1/m}} = O_m(1)
  \end{equation}
  and collecting all asymptotic constants into $C_1$ we arrive to
  contradiction with~\eqref{eq:hypersurface-cond}.
\end{proof}

\begin{Cor}\label{cor:hypersurface-select}
  Suppose $f_i\in\cO(\bar\Delta')$ with $\norm{f_i}_{\Delta'}\le M$.
  Let
  \begin{equation}
    Y = \vf(X\cap D) \subset \C^{m+1}.
  \end{equation}
  For every $\e>0$ there exist two positive constants
  \begin{align}
     d&=d(m,\e,e(\cD)) \\
     C&=C(m,\e,e(\cD),\norm\cD,M)
  \end{align}
  such that if $\delta\le CH^{-\e}$ then $Y(\Q,H)$ is contained in an
  algebraic hypersurface of degree at most $d$ in $\C^{m+1}$.
\end{Cor}
\begin{proof}
  By Proposition~\ref{prop:hypersurface-select} for $m>0$ it is enough
  to choose
  \begin{gather}
    \e > \frac{C_1}{(d/e(\cD))^{1/m}}, \\
    -\log C > \frac{C_1 d^{-1}\log(\norm\cD e(\cD) M)}{(d/e(\cD))^{1/m}},
  \end{gather}
  and for $m=0$ it is enough to choose $d=e(\cD)$ and e.g. $C=1$.
\end{proof}

\section{Exploring rational points}
\label{sec:exploring}

We begin with a definition.

\begin{Def}\label{def:xw}
  Let $X\subset\C^m$ and $W\subset\C^m$ be two sets. We define
  \begin{equation}
    X(W) := \{ w\in W : W_w \subset X \}
  \end{equation}
  to be the set of points of $W$ such that $X$ contains the germ of
  $W$ around $w$, i.e. such that $w$ has a neighborhood
  $U_w\subset\C^m$ such that $W\cap U_w\subset X$.
\end{Def}

If $A\subset\C^n$ we denote by $A_\R:=A\cap\R^n$. We
remark that
\begin{equation}\label{eq:alg-part-R}
  (A(W))_\R\subset(A_\R)(W_\R).
\end{equation}
We will consider Definition~\ref{def:xw} in two cases: for
$X\subset\C^m$ locally analytic and $W\subset\C^m$ an algebraic
variety, and for $X\subset\R^m$ subanalytic and $W\subset\R^m$ a
semi-algebraic set.

Our principal motivation for Definition~\ref{def:xw} is the following
direct consequence (cf. Theorem~\ref{thm:subanalytic-main}).
\begin{Lem}\label{lem:alg-part-trans}
  Let $S\subset\R^m$ be a connected positive-dimensional semialgebraic
  set and $A\subset\R^m$. Then $A(S)\subset A^\alg$.
\end{Lem}

We record some simple consequences.
\begin{Lem}\label{lem:alg-part-rules}
  Let $A,B,W\subset\C^m$. Then
  \begin{equation}
    A(W)\cup B(W)\subset(A\cup B)(W).
  \end{equation}
  If $A\subset B$ is relatively open then
  \begin{equation}
    B(W)\cap A = A(W).
  \end{equation}
\end{Lem}

\subsection{Projections from admissible graphs}
\label{sec:admissible-graph}

Let $\Omega_z\subset\C^m,\Omega_w\subset\C^n$ be domains and set
$\Omega:=\Omega_z\times\Omega_w\subset\C^{m+n}$. Let $\Lambda$ be an
analytic space. We denote by $\pi_z,\pi_w,\pi_\Lambda$ the projections
from $\Omega\times\Lambda$ to $\Omega_z,\Omega_w,\Lambda$
respectively. We denote by
$\pi:\Omega\times\Lambda\to\Omega_z\times\Lambda$ the projection
$\pi=\pi_z\times\pi_\Lambda$.

Let $U\subset\Omega_z\times\Lambda$ be an open subset and
$\psi:U\to\Omega_w$ a function, and denote its graph by
\begin{equation}
  \Gamma_\psi := \{ (z,w,\l)\in\Omega\times\Lambda : \psi(z,\l)=w \}.
\end{equation}
We denote by $\tilde\psi:U\to\Gamma_\psi$ the map
$(z,\l)\to(z,\psi(z,\l),\l)$.

\begin{Def}\label{def:admissible-graph}
  We say that $\psi:U\to\Omega_w$ is \emph{admissible} if
  $\Gamma=\Gamma_\psi$ is relatively compact in $\Omega\times\Lambda$,
  and if there exists an analytic subset
  $X_\Gamma\subset\Omega\times\Lambda$ which agrees with $\Gamma$ over
  $U$, i.e. $X_\Gamma\cap\pi^{-1}(U)=\Gamma$.
\end{Def}

\subsection{Rational points on admissible projections}

For the remainder of this section we fix an admissible
$\phi:U\to\Omega_w$. Our main result in this section is the following
theorem.

\begin{Thm}\label{thm:exploring-projection}
  Let $X\subset\Omega\times\Lambda$ be an analytic family. Set
  \begin{equation}
    Y:=\pi(X\cap\Gamma)\subset\Omega_z\times\Lambda.
  \end{equation}
  Let $\e>0$. There exist constants $d=d(\Gamma,\e)$ and
  $N= N(\Gamma,\e)$ with the following property. For any $\l\in\Lambda$
  and any $H\in\N$ there exist at most $N H^\e$ many irreducible
  algebraic varieties $V_\alpha\subset\C^m$ with $\deg V_\alpha\le d$
  such that
  \begin{equation}
    Y_\l(\Q,H)\subset \bigcup_\alpha Y_\l(V_\alpha).
  \end{equation}
\end{Thm}

We begin the proof of Theorem~\ref{thm:exploring-projection} with the
following proposition.

\begin{Prop}\label{prop:exploring-family}
  Let $X\subset\Omega\times\Lambda$ be an analytic family
  and set
  \begin{equation}
    Y:=\pi(X\cap\Gamma)\subset\Omega_z\times\Lambda.
  \end{equation}
  Let $W\subset\C^m$ be an irreducible algebraic variety.

  Let $\e>0$. There exist constants $d = d(\Gamma,\e,\deg W)$ and
  $N = N(\Gamma,\e,\deg W)$ with the following property. For any
  $\l\in\Lambda$ and any $H\in\N$ there exist $N H^\e$ hypersurfaces
  $H_\alpha\subset\C^m$ with $\deg H_\alpha\le d$ such that
  $W\not\subset H_\alpha$ and
  \begin{equation}
    (Y_\l\cap W)(\Q,H)\subset Y_\l(W)\cup \bigcup_\alpha H_\alpha.
  \end{equation}
\end{Prop}
\begin{proof}
  Since the statement involves only the intersection $X\cap\Gamma$
  we may without loss of generality replace $X$ by $X\cap X_\Gamma$,
  and assume that $X\cap\pi^{-1}(U)\subset\Gamma$.
  
  Set $k:=\dim W$. Let $\cC$ denote the projective (hence compact)
  Chow variety (see \cite[Chapter~4]{gkz}) parametrizing all effective
  algebraic cycles of dimension $k$ and degree equal to $\deg W$. We
  denote by $R_V\in\cC$ the point corresponding to a cycle $V$. Then
  the following family is analytic,
  \begin{equation}
    X'\subset\Omega\times(\Lambda\times\cC), \qquad
    X'=\{(z,w,\l,R_V):(z,w,\l)\in X, z\in\supp V\}.
  \end{equation}
  Clearly
  \begin{equation}
    X'_{(\l,\cR_W)} = X_\l \cap(W\times\Omega_w).
  \end{equation}
  We apply Theorem~\ref{thm:decomp} to $X'$ with the compact set
  $\bar\Gamma\times\cC$ and consider the conclusion for points of the
  form $(p,\l,R_W)$ for $(p,\l)\in\bar\Gamma$. We conclude that there
  exist $r,C_D,C_H>0$ depending only on $\Gamma,\deg W$ such that for
  any $(p,\l)\in\bar\Gamma$ there exists a decomposition datum $\cD$
  satisfying
  \begin{enumerate}
  \item $\Delta=\Delta'$ is centered at $p$, and
    $B_r(p)\subset\Delta\subset\Omega$.
  \item $\dim\cM_i< k$, $\norm\cD\le C_D$ and $e(\cD)\le C_H$.
  \item The set  
    \begin{equation}
      Z_\l = (X_\l\cap (W\times\Omega_w))^{<k}.
    \end{equation}
    admits decomposition with respect to $\cD$.
  \end{enumerate}

  Fix $\l\in\Lambda$, let $q\in Y_\l\cap W$ and suppose
  $q\not\in\Sing W$ and $q\not\in Y_\l(W)$. Then the germ of $W$ at
  $q$ is smooth $k$-dimensional and not contained in $Y_\l$.
  Equivalently its image $\tilde\psi(W\times\{\l\})\subset\Gamma$ is
  the germ of a smooth $k$-dimensional analytic set at
  $\tilde\psi(q,\l)$ which is not contained in $X_\l$. Since we assume
  $X\subset\Gamma$ in a neighborhood of $\tilde\psi(q,\l)$ we conclude
  that the dimension of
  \begin{equation}
    X_\l\cap (W\times\Omega_w) = X_\l\cap\Gamma_\l\cap(W\times\Omega_w)
    = X_\l\cap\tilde\psi(W\times\{\l\})
  \end{equation}
  at $\tilde\psi(q,\l)$ is strictly smaller than $k$, i.e.
  $(q,\psi(q,\l))\in Z_\l$ and thus $q\in\pi(Z_\l)$. In conclusion,
  \begin{equation}\label{eq:Y_l-desc}
    Y_\l\cap W \subset Y_\l(W)\cup\Sing W\cup \pi(Z_\l).
  \end{equation}
  
  Fix a hypersurface $H_0\subset\C^m$ containing $\Sing W$ and not
  containing $W$. It is clear that one can choose $H_0$ of some degree
  $d_0$ depending only on $\deg W$.
  
  Since $\dim W=k$, one can choose a subset of $k$ coordinates on
  $\C^m$, say $\vf=(z_1,\ldots,z_k)$, such that $\vf:W\to\C^k$ is
  dominant. In particular, no non-zero polynomial in the coordinates
  $\vf$ vanishes on $W$. Since $\bar\Gamma$ is compact, the
  coordinates $\vf$ are certainly bounded (in absolute value) in the
  $r$-neighborhood of $\bar\Gamma$ by some number $M$. Fix some
  $\e'>0$ whose value will be determined later, and let
  \begin{align}
     d&=d(m,\e',C_H) \\
     C&=C(m,\e',C_H,C_D,M)
  \end{align}
  be the two constants of Corollary~\ref{cor:hypersurface-select}.

  Let $\delta=CH^{-\e'}$. Let $p\in(\bar\Gamma)_\l$ and denote
  $D_p=D_{\delta r}(p)\subset\Omega$. We apply
  Corollary~\ref{cor:hypersurface-select} to $\vf$ and the set $Z_\l$
  and conclude that there exists a polynomial $P_p(\vf)$ of degree at
  most $d$ such that
  \begin{equation}
    (\pi_z(D_p\cap Z_\l))(\Q,H) \subset \{P_p=0\}.
  \end{equation}
  We let $H_p:=\{P_p=0\}\subset\C^m$.
  
  Finally it remains to cover the compact set $(\bar\Gamma)_\l$ by the
  polydiscs $\{D_p:p\in S\}$ for some finite set $S\subset(\bar\Gamma)_\l$
  and take
  \begin{equation}
    \{H_\alpha\}=\{H_0\}\cup \{H_p:p\in S\}.
  \end{equation}
  Then~\eqref{eq:Y_l-desc} and the choice of $H_0,H_p$ gives
  \begin{equation}
    (Y_\l\cap W)(\Q,H)\subset Y_\l(W)\cup \bigcup_\alpha H_\alpha.
  \end{equation}
  as claimed.
  
  Since each $D_p$ has radius at least $CrH^{-\e'}$ and $\bar\Gamma$ is
  compact, it is easy to see that one can choose a covering of size at
  most $N H^{\e'(n+m)}$, where $N=N(Cr,\Gamma,\e')$. Finally taking $\e'=\e/(n+m)$
  we obtain the statement of the proposition.
\end{proof}

The following Lemma gives an inductive proof of
Theorem~\ref{thm:exploring-projection}, which is obtained for the case
$W=\C^m$.

\begin{Lem}\label{lem:exploring-induction}
  Let $X\subset\Omega\times\Lambda$ be an analytic family
  and set
  \begin{equation}
    Y:=\pi(X\cap\Gamma)\subset\Omega_z\times\Lambda.
  \end{equation}
  Let $W\subset\C^n$ be an irreducible algebraic variety.
  
  Let $\e>0$. There exist constants $d = d(\Gamma,\e,\deg W)$ and
  $N = N(\Gamma,\e,\deg W)$ with the following property. For any
  $\l\in\Lambda$ and any $H\in\N$ there exist at most $N H^\e$ many
  irreducible algebraic varieties $V_\alpha\subset\C^m$ with
  $\deg V_\alpha\le d$ such that
  \begin{equation}\label{eq:exploring-induction}
    (Y_\l\cap W)(\Q,H)\subset \bigcup_\alpha Y_\l(V_\alpha).
  \end{equation}
\end{Lem}
\begin{proof}
  We proceed by induction on $\dim W$. Apply
  Proposition~\ref{prop:exploring-family} to obtain a family of at
  most $N'H^{\e/2}$ hypersurfaces $\{H_{\alpha'}\}\subset\C^m$ of
  degree $d'$. Let $\{W_\alpha\}$ denote the union over $\alpha'$ of
  the sets of irreducible components of $W\cap H_{\alpha'}$. Then
  \begin{align}
    \#\{W_\alpha\}&\le d'N'H^{\e/2}, & \dim W_\alpha&=\dim W-1, & \deg W_\alpha &\le d'\cdot\deg W
  \end{align}
  and   
  \begin{equation}\label{eq:exploring-induction-1}
    (Y_\l\cap W)(\Q,H)\subset Y_\l(W)\cup \bigcup_\alpha W_\alpha.
  \end{equation}
  Apply the inductive hypothesis to each $W_\alpha$ to obtain
  collections $W_{\alpha,\beta}$, of size at most $N''H^{\e/2}$ for
  each $\alpha$, such that
  \begin{equation}\label{eq:exploring-induction-2}
    (Y_\l \cap W_\alpha)(\Q,H) \subset \bigcup_\beta W_{\alpha,\beta}.
  \end{equation}
  Finally we take $\{V_\alpha\}$ to be the union of the sets $\{W\}$
  and $\{W_{\alpha,\beta}\}$. The size of $\{V_\alpha\}$ is bounded by
  $1+d'N'N''H^\e$ as claimed and~\eqref{eq:exploring-induction} is
  satisfied by~\eqref{eq:exploring-induction-1}
  and~\eqref{eq:exploring-induction-2}.
\end{proof}

\section{Subanalytic sets and $L^D_\an$}
\label{sec:subanalytic}

Let $I=[-1,1]$. For $m\ge0$ we let $\R\{X_1,\ldots,X_m\}$ denote the
ring of power series converging in a neighborhood of $I^m$.
To each $f\in\R\{X_1,\ldots,X_m\}$ we naturally associate
the map $f:I^m\to\R$.

We recall the language $L^D_\an$ of \cite{denef-vdd}. The language
includes a countable set of variables $\{X_1,X_2,\ldots\}$, a
relation symbol $<$, a binary operation symbol $D$, and an $m$-ary
operation symbol $f$ for every $f\in\R\{X_1,\ldots,X_m\}$ satisfying
$f(I^m)\subset I$. We view $I$ as an $L^D_\an$-structure by
interpreting $<$ and $f$ in the obvious way and interpreting $D$ as
\emph{restricted division}, namely
\begin{equation}
  D(x,y) =
  \begin{cases}
    x/y & |x|\le|y| \text{ and } y\neq0 \\
    0 & \text{otherwise.}
  \end{cases}
\end{equation}
We denote by $L_\an$ the language obtained from $L_\an^D$ by omitting
$D$.

For every $L_\an^D$-term $t(X_1,\ldots,X_m)$ we have an associated map
$t:I^m\to I$ which we denote $x\to t(x)$. If $t$ is an $L_\an$-term
then this map is real analytic in $I^m$.

For an $L_\an^D$-formula $\phi(X_1,\ldots,X_m)$ we write $\phi(I^m)$
for the set of points $x\in I^m$ satisfying $\phi$. If $A\subset I^m$
we write $\phi(A):=\phi(I^m)\cap A$. We will use the following key
result of \cite{denef-vdd}.
\begin{Thm}\label{thm:denef-vdd}
  $I$ has elimination of quantifiers in $L^D_\an$. As a consequence, a
  set $A\subset I^m$ is subanalytic in $\R^m$ if an only if it is
  defined by a quantifier-free formula $\phi$ of $L^D_\an$.
\end{Thm}

\subsection{Admissible formulas}

Let $U\subset\mathring{I}^m$ be an open subset. We define the notion
of an $L_\an^D$-term \emph{admissible} in $U$ by recursion as follows:
a variable $X_j$ is always admissible in $U$; a term
$f(t_1,\ldots,t_m)$ is admissible in $U$ if and only if the terms
$t_1,\ldots,t_m$ are admissible in $U$; and a term $D(t_1,t_2)$ is
admissible in $U$ if $t_1,t_2$ are admissible in $U$ and if
\begin{equation}
  |t_1(x)| \le |t_2(x)| \text{ and } t_2(x)\neq0
\end{equation}
for every $x\in U$. An easy induction gives the following.
\begin{Lem}\label{lem:admissible-term-analytic}
  If $t$ is admissible in $U$ then the map $t:U\to I$ is real
  analytic.
\end{Lem}

We will say that an $L_\an^D$-formula $\phi$ is admissible in $U$ if
all terms appearing in $\phi$ are admissible in $U$. Here and below,
when speaking about ``terms appearing in $\phi$'' we consider not only
the top-level terms appearing in the relations, but also every
sub-term appearing in the construction tree of each term. The
following proposition shows that when considering definable subsets of
$I$ one can essentially reduce to admissible formulas.

\begin{Prop}\label{prop:admissible-decomp}
  Let $U\subset\mathring I^m$ be an open subset and
  $\phi(X_1,\ldots,X_m)$ a quantifier-free $L_\an^D$-formula. There
  exist open subsets $U_1,\ldots,U_k\subset U$ and quantifier-free
  $L_\an^D$-formulas $\phi_1,\ldots,\phi_k$ such that $\phi_j$ is
  admissible in $U_j$ and
  \begin{equation}
    \phi(U) = \bigcup_{j=1}^k \phi_j(U_j).
  \end{equation}
\end{Prop}
\begin{proof}
  We prove the claim by induction on the number $N$ of
  $U$-inadmissible terms in $\phi$. Clearly if this number is zero we
  are done. Otherwise let $t$ be some minimal $U$-inadmissible term in
  $\phi$, i.e. such that all sub-terms appearing in $t$ are
  $U$-admissible. Express $\phi$ in the form
  \begin{equation}
    \phi(X_1,\ldots,X_m) = \phi'(X_1,\ldots,X_m,t)
  \end{equation}
  for a quantifier-free $L_\an^D$-formula $\phi'$ of $m+1$ free
  variables. By the definition of admissibility it is clear that
  $t=D(t_1,t_2)$, and by minimality $t_1,t_2$ are admissible in $U$.
  We let
  \begin{align}
    \phi_1&\equiv\phi & U_1 &= \{ x\in U: |t_1(x)|<|t_2(x)| \}
  \end{align}
  and
  \begin{align}
    \phi_2&=(|t_1|=|t_2|)\land (t_2\neq0)\land \phi'(X_1,\ldots,X_m,1) \\
    \phi_3&=(|t_1|>|t_2| \lor (t_2=0))\land \phi'(X_1,\ldots,X_m,0) .
  \end{align}
  with $U_2=U_3=U$. Note that for readability we use the absolute value
  as a shorthand above, but it is clear that the relations can be
  expressed in terms of the unary minus operation corresponding to the
  function $-:I\to I, x\to(-x)$.

  Since $t_1,t_2$ are admissible in $U$ they define real analytic (in
  particular continuous) functions there, and the relation defining
  $U_1$ is indeed open. Moreover in $U_1$ the term $t$ is admissible
  by definition and hence the number of $U_1$-inadmissible terms in
  $\phi_1$ is strictly smaller than $N$. Similarly, the new relations
  introduced in $\phi_2$ (resp. $\phi_3$) are admissible in $U$, and
  since the inadmissible $t$ is replaced by the admissible term $1$
  (resp. $0$) the number of $U_2$ (resp. $U_3$) inadmissible terms
  is strictly smaller than $N$. It is an easy exercise to check
  that
  \begin{equation}
    \phi(U) = \bigcup_{j=1}^3 \phi_j(U_j).
  \end{equation}
  The proof is now concluded by applying the inductive hypothesis to
  each pair $\phi_j,U_j$ for $j=1,2,3$.
\end{proof}

\subsection{Basic formulas and equations}

We say that $\phi$ is a \emph{basic $D$-formula} if it has
the form
\begin{equation}
  \big(\land_{j=1}^k t_j(X_1,\ldots,X_m)=0\big)\land\big(\land_{j=1}^{k'} s_j(X_1,\ldots,X_m)>0\big)
\end{equation}
where $t_j,s_j$ are $L_\an^D$-terms. It is easy to check the following.

\begin{Lem}\label{lem:basic-disjunction}
  Every quantifier free $L_\an^D$-formula $\phi$ is equivalent in the
  structure $I$ to a finite disjunction of basic formulas. If $\phi$
  is $U$-admissible then so are the basic formulas in the disjunction.
\end{Lem}

We say that $\phi$ is a \emph{basic $D$-equation} if $k'=0$, i.e. if
it involves only equalities. If $\phi$ is a basic $D$-formula we
denote by $\tilde\phi$ the basic $D$-equation obtained by removing all
inequalities.

Let $\phi$ be a $U$-admissible basic $D$-formula for some
$U\subset I^m$. Then $\tilde\phi$ is $U$-admissible as well. Moreover
since all the terms $s_j$ evaluate to continuous functions in $U$ the
strict inequalities of $\phi$ are open in $U$ and we have the
following.

\begin{Lem}\label{lem:phi-tilde-rel-open}
  Suppose $\phi$ is $U$-admissible basic $D$-formula. Then $\phi(U)$
  is relatively open in $\tilde\phi(U)$.
\end{Lem}

The set defined by an admissible $D$-equation can be described in
terms of admissible projections in the sense
of~\secref{sec:admissible-graph}.

\begin{Prop}\label{prop:D-equation-to-proj}
  Let $U\subset\mathring I^{m+l}$ and $\phi$ be a $U$-admissible $D$-equation,
  \begin{equation}
    \phi = (t_1=0)\land\cdots\land(t_k=0).
  \end{equation}
  In the notations of~\secref{sec:admissible-graph}, there exist
  \begin{enumerate}
  \item Complex domains
    \begin{align}
      \Omega_z&\subset\C^m, & \Omega_w&\subset\C^N, & \Lambda&\subset\C^l
    \end{align}
    with $N\in\N$ and $I^{m+l}\subset\Omega_z\times\Lambda$.
  \item An open complex neighborhood
    $U\subset U_\C\subset \Omega_z\times\Lambda$.
  \item An analytic map $\psi:U_\C\to\Omega_w$.
  \item An analytic set $X\subset \Omega\times\Lambda$.
  \end{enumerate}
  such that $\psi$ is admissible and $Y:=\pi(X\cap\Gamma_\psi)$
  satisfies $Y_\R=\phi(U)$.
\end{Prop}
\begin{proof}
  Let $\{s_1,\ldots,s_N\}$ denote all terms of the type
  $s_j=D(s_{j,1},s_{j,2})$ appearing in $\phi$. As a notational
  convenience we write $X=X_1,\ldots,X_{m+n}$ and $W=W_1,\ldots,W_N$
  for variables on $\R^{n+m}$ and $\R^N$.
  
  For every term $t(X)$ appearing in $\phi$ we define an $L_\an$ term
  $t'(X,W)$ by recursion as follows: if $t$ is a variable then
  $t':=t$; if $t=f(t_1,\ldots,t_k)$ then $t'=f(t_1',\ldots,t_k')$;
  finally if $t=s_j$ then $t'=W_j$. As $L_\an$-terms, every term $t'$
  corresponds to a real-analytic function $t':I^{m+N+n}\to I$. We let
  $\Omega_z\times\Omega_w\times\Lambda$ denote some complex
  neighborhood of $I^{m+N+n}$ to which every term $t'$ admits analytic
  continuation.
  
  By Lemma~\ref{lem:admissible-term-analytic}, all terms appearing in
  $\phi$ evaluate to real analytic maps from $U$ to $I$. We define a
  map $\psi:U\to\Omega_w$ by
  \begin{equation}
    \psi(x) = (s_1(x),\ldots,s_N(x)).
  \end{equation}
  We note that $\psi(U)\subset I^N$ and, by definition of
  $U$-admissibility, all the terms $s_{j,2}$ evaluate to non-vanishing
  functions on $U$. Let $U_w$ be a relatively compact neighborhood of
  $I^N$ in $\Omega_w$. Then there exists a relatively compact open
  neighborhood $U_\C\subset\Omega_z\times\Lambda$ of $U$ such that
  \begin{enumerate}
  \item $(U_\C)_\R=U$ and $\psi(U_\C)\subset U_w$.
  \item All terms appearing in $\phi$ admit analytic continuation to
    $U_\C$.
  \item All the terms $s_{j,2}$ evaluate to non-vanishing maps on $U_\C$.
  \end{enumerate}
  Henceforth we view $U_\C$ as the
  domain of $\psi$. The graph $\Gamma=\Gamma_\psi$ is relatively
  compact in $\Omega\times\Lambda$, being contained in the product of
  the relatively compact sets $U_\C\subset\Omega_z\times\Lambda$ and
  $U_w\subset\Omega_w$.

  By construction of $t'$ it is clear that
  \begin{equation}\label{eq:t-v-t'}
    t(z,\l) = t'(z,\psi(z,\l),\l) \qquad \text{for }(t,\l)\in U_\C.
  \end{equation}
  We define the analytic subset
  $X_\Gamma\subset\Omega\times\Lambda$ by
  \begin{equation}
    X_\Gamma = \{ s_{j,2}' W_j = s_{j,1}' : j=1,\ldots,N \}.
  \end{equation}
  It is easy to check by induction that $X_\Gamma$ agrees with
  $\Gamma$ over $U_\C$ (using~\eqref{eq:t-v-t'} and the fact that
  $s_{j,2}$ evaluate to non-vanishing maps on $U_\C$).

  Finally we define $X\subset\Omega\times\Lambda$ by
  \begin{equation}
    X = \{ t_1' = \cdots = t_k'= 0 \}
  \end{equation}
  and set $Y:=\pi(X\cap\Gamma)\subset\Omega_z\times\Lambda$.
  Then~\eqref{eq:t-v-t'} at the points of $U=(U_\C)_\R$ gives
  $Y_\R=\phi(U)$.
\end{proof}

\subsection{Estimate for subanalytic sets}

To state the general form of our main result we introduce the notion
of \emph{complexity} of a semi-algebraic set. We say that a
semialgebraic set $S\subset\R^m$ has complexity $(m,s,d)$ if it
defined by a semialgebraic formula involving $s$ different relations
$P_j>0$ or $P_j=0$, where the polynomials $P_j$ have degrees bounded
by $d$.

\begin{Thm}\label{thm:subanalytic-main}
  Let $A\subset\R^{m+n}$ be a bounded subanalytic set and
  $\e>0$. There exist constants $d=d(A,\e)$ and $N=N(A,\e)$ with the
  following property. For any $y\in I^n$ and any $H\in\N$ there exist
  at most $N H^\e$ many smooth connected semialgebraic sets
  $S_\alpha\subset\C^m$ with complexity $(m,d,d)$ such that
  \begin{equation}
    A_y(\Q,H)\subset \bigcup_\alpha A_y(S_\alpha).
  \end{equation}
\end{Thm}
\begin{proof}
  Rescaling $A$ by a sufficiently large integer, $\tilde A:=\frac1M A$
  we may assume that $\tilde A\subset\mathring I^{n+m}$. This clearly
  does not affect the heights of points by more than a constant
  factor, and it will suffice to prove the claim for $\tilde A$ and
  then rescale the varieties $\tilde V_\alpha$ back to
  $V_\alpha:=M V_\alpha$. We thus assume without loss of generality
  that $A\subset\mathring I^{n+m}$.
  
  By Theorem~\ref{thm:denef-vdd} we may write
  $A=\phi(\mathring I^{m+n})$ for some quantifier-free
  $L^D_\an$-formula $\phi$. By
  Proposition~\ref{prop:admissible-decomp} and
  Lemma~\ref{lem:basic-disjunction} we may write
  \begin{equation}
    A = \bigcup_{j=1}^k \bigcup_{i=1}^{n_j} \phi_{ji}(U_j)
  \end{equation}
  where $\phi_{ji}$ is a $U_j$-admissible basic $D$-formula. By the
  first part of Lemma~\ref{lem:alg-part-rules} it is clear that it
  will suffice to prove the claim with $A$ replaced by each
  $\phi_{ij}(U_j)$. We thus assume without loss of generality that
  $\phi$ is already a $U$-admissible basic $D$-formula and prove the
  claim for $A=\phi(U)$.

  Recall that $\tilde\phi$ is a $U$-admissible $D$-equation. We write
  $B=\tilde\phi(\mathring I^{m+n})$. Applying
  Proposition~\ref{prop:D-equation-to-proj} to $\tilde\phi$ and using
  Theorem~\ref{thm:exploring-projection} we construct a locally
  analytic set $Y\subset\C^{m+n}$ such that $Y_\R=B$, and for any
  $y\in I^n$ there exist at most $N' H^\e$ many irreducible algebraic
  varieties $V_\alpha\subset\C^m$ with $\deg V_\alpha\le d'$ such that
  \begin{equation}
    Y_y(\Q,H)\subset \bigcup_\alpha Y_y(V_\alpha)
  \end{equation}
  with $d,N$ as in Theorem~\ref{thm:exploring-projection}.
  By~\eqref{eq:alg-part-R} and $Y_\R=B$ we have
  \begin{equation}
    B_y(\Q,H)\subset \bigcup_\alpha B_y(\cV_\alpha), \qquad \cV_\alpha=(V_\alpha)_\R.
  \end{equation}
  Finally, we recall that $A$ is relatively
  open in $B$ by Lemma~\ref{lem:phi-tilde-rel-open}. Then the same is
  true for the fiber $A_y\subset B_y$, and
  \begin{equation}
    \begin{split}
      A_y(\Q,H)&\subset A_y\cap(B_y(\Q,H))\subset A_y\cap
      \bigcup_\alpha B_y(\cV_\alpha) = \bigcup_\alpha ( A_y\cap
      B_y(\cV_\alpha)) \\ &= \bigcup_\alpha A_y(\cV_\alpha)
    \end{split}
  \end{equation}
  where the last equality is given by the second part of
  Lemma~\ref{lem:alg-part-rules}.

  If we write each real-algebraic variety $\cV_\alpha$ as a union of
  smooth connected strata $\cV_\alpha=\cup_j S_{\alpha,j}$ then we have
  \begin{equation}
    A_y(\Q,H)\subset\bigcup_\alpha A(\cV_\alpha) \subset \bigcup_{\alpha,j} A(S_{\alpha,j}).
  \end{equation}
  It remains to note that since $\deg V_\alpha\le d'$, the number and
  complexity of the strata $S_{\alpha,j}$ is bounded by some number
  $d$ depending only on $d'$. This follows from general uniformity
  properties in the algebraic category, and in fact one may derive
  explicit (and polynomial in $d'$) estimates for $d$ (for details see
  \cite[Proposition~37]{me:rest-wilkie}). Taking $N=N'd$ finishes the
  proof.
\end{proof}

We can now finish the proof of Theorem~\ref{thm:main}.
\begin{proof}[Proof of Theorem~\ref{thm:main}.]
  By Theorem~\ref{thm:subanalytic-main} for any $y\in I^n$ and any
  $H\in\N$ there exist at most $N H^\e$ many smooth connected
  semialgebraic sets $S_\alpha\subset\C^m$ such that
  \begin{equation}
    A_y(\Q,H)\subset \bigcup_\alpha A_y(S_\alpha).
  \end{equation}
  By Lemma~\ref{lem:alg-part-trans}, for any positive dimensional
  $S_\alpha$ we have $A(S_\alpha)\subset A_y^\alg$. Thus
  $\#A_Y^\trans(\Q,H)$ is bounded by the number of zero-dimensional
  strata, i.e. by $NH^\e$.
\end{proof}

\bibliographystyle{plain} \bibliography{../rest-wilkie/nrefs}

\end{document}